\documentclass[12pt]{article}
\usepackage{amssymb, amscd, amsthm, amsmath,latexsym, amstext,mathrsfs,verbatim}
\usepackage[all]{xy}
\usepackage{graphicx,enumerate}
\usepackage{url}

\def\ddA{{\rm A}}
\def\ddD{{\rm D}}

\def\Br{{\rm Br}}

\def\BrM{{\rm BrM}}

\def\ddB{{\rm B}}
\def\ddC{{\rm C}}
\def\ddD{{\rm D}}
\def\ddE{{\rm E}}
\def\ddF{{\rm F}}
\def\ddG{{\rm G}}
\def\ddH{{\rm H}}
\def\ddI{{\rm I}}

\newcommand{\cA}{\mathcal{A}}

\newcommand{\R}{\mathbb R}
\newcommand{\Z}{\mathbb Z}

\numberwithin{equation}{section}

\newtheorem{lemma}{Lemma}[section]
\newtheorem{cor}[lemma]{Corollary}

\newtheorem{thm}[lemma]{Theorem}

\theoremstyle{definition}

\newtheorem{defn}[lemma]{Definition}

\theoremstyle{remark}
\newtheorem{rem}[lemma]{Remark}

\begin{document}
\title{Brauer algebras of type $\ddH_3$ and $\ddH_4$}
\author{Shoumin Liu}
\date{}
\maketitle

\begin{abstract}
In this paper, we will present Brauer algebras associated to spherical  Coxeter groups of type $\ddH_3$ and $\ddH_4$, which are also can be regarded
as subalgebras of Brauer algebras  $\ddD_6$ and $\ddE_8$ by M\"uhlherr's admissible partition. Also some basic
properties will be described here.
 \end{abstract}

\section{Introduction}
From studying the invariant theory for orthogonal groups,
Brauer discovered algebras which are now called  Brauer algebras of type $\ddA$ in \cite{Brauer1937};  Cohen, Frenk and Wales extended it to the
definition of simply laced type in \cite{CFW2008}, the nodes of whose  Dynkin diagrams are connected by simple bond.
M\"{u}hlherr  described how to get Coxeter group of type $\ddH_3$($\ddH_4$) by twisting Coxeter group of  type $\ddD_6$($\ddE_8$) in \cite{M1992}.
Here we will apply a similar approach as M\"{u}hlherr on $\Br(\ddD_6)$($\Br(\ddE_8)$), to get an algebra $\Br(\ddH_3)$ ($\Br(\ddH_4)$) called the Brauer algebra of type $\ddH_3$($\ddH_4$).

In fact, M\"{u}hlherr's method can be considered as the generalization of obtaining Weyl groups of non-simply laced types from simply-laced types(\cite{Car}, \cite{T1959}), such as $\ddC_n$ from $\ddA_{2n-1}$, $\ddB_n$ from $\ddD_{n+1}$, $\ddF_4$ from $\ddE_6$. These are motivated by considering the invariant subgroups
under non-trivial  automorphisms by action on their Dynkin  diagrams. We have already utilized it to obtain Brauer algebras of type $\ddC_n$(\cite{CLY2010}),
$\ddB_n$(\cite{CL2012}), $\ddF_4$(\cite{L2013}).
This paper can be regarded as a part of the project of finding Brauer
algebras of non-simply laced types from simply-laced types.

The diagrams of $\ddH_3$, $\ddH_4$, $\ddD_6$, and $\ddE_8$ are presented below,
and the diagrams of $\ddD_6$ and $\ddE_8$ are specially depicted for M\"uhlherr's admissible partition corresponding to the diagram of
$\ddH_3$ and $\ddH_4$.
\begin{figure}
\begin{center}
\includegraphics[width=.7\textwidth,height=.2\textheight]{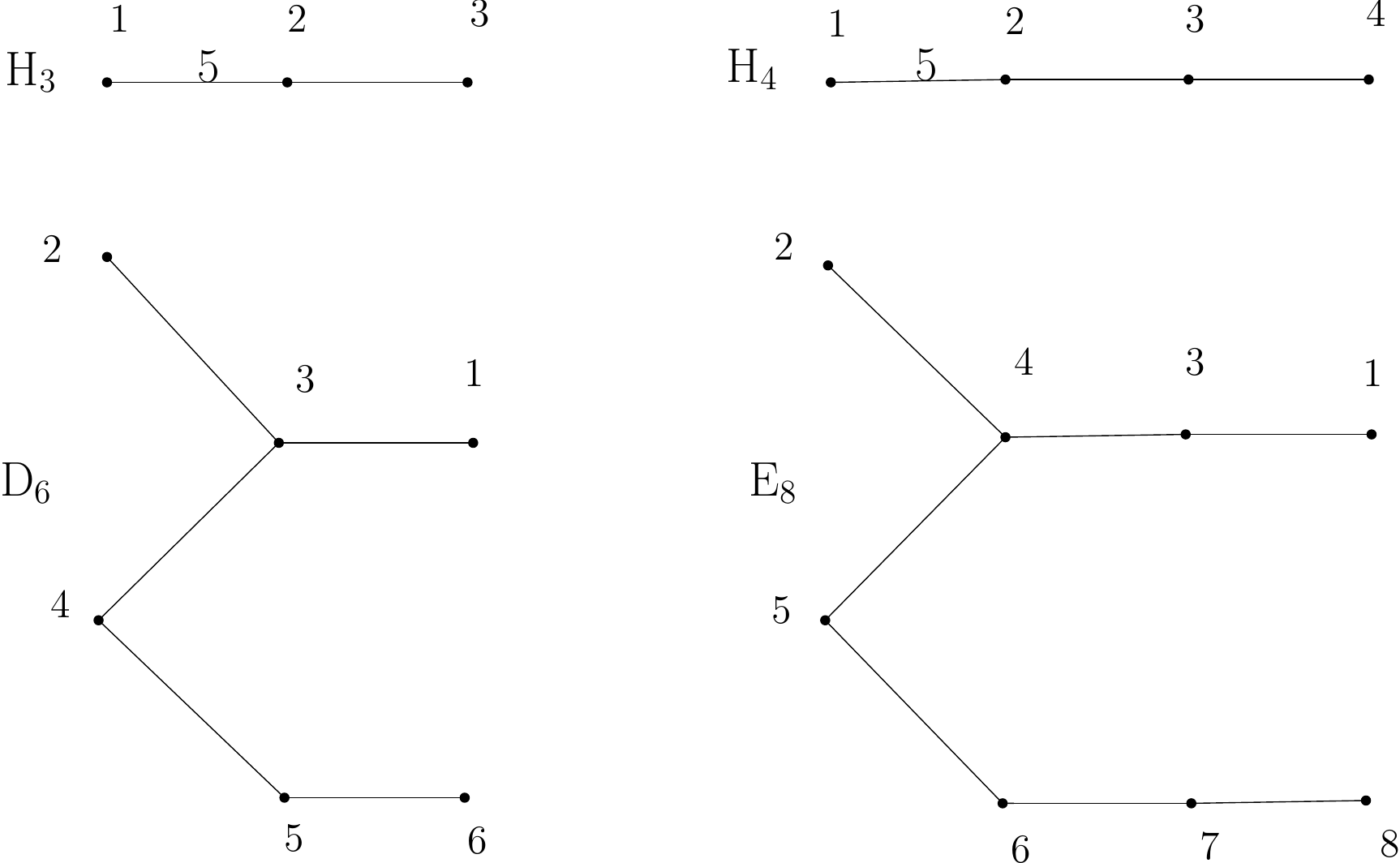}
\end{center}
\caption{Coxeter diagrams of $\ddH_3$,$\ddH_4$, $\ddD_6$ and $\ddE_8$}
\label{H3H4}
\end{figure}

In this paper, we will present the following two main theorems about $\Br(\ddH_3)$ and $\Br(\ddH_4)$, respectively.
To avoid confusion, the generators of $\Br(\ddD_6)$ and $\Br(\ddE_8)$ are capitalized.
\begin{thm}\label{mainthm1} There exists an injective  $\Z[\delta^{\pm 1}]$-algebra homeomorphism
   $$\phi_1:\Br(\ddH_3)\longrightarrow \Br(\ddD_6)$$
   determined by $\phi_1(r_1)=R_2R_4$, $\phi_1(r_2)=R_3R_5$, $\phi_1(r_3)=R_1R_6$, $\phi_1(e_1)=E_2E_4$,
   $\phi_1(e_2)=E_3E_5$ and  $\phi_1(e_3)=E_1E_6$. Furthermore $\Br(\ddH_3)$ are free of rank $1045$ over
   $\Z[\delta^{\pm 1}]$.
 \end{thm}

\begin{thm}\label{mainthm2} There exists an injective  $\Z[\delta^{\pm 1}]$-algebra homeomorphism
   $$\phi_2:\Br(\ddH_4)\longrightarrow \Br(\ddE_8)$$
   determined by $\phi_2(r_1)=R_2R_5$, $\phi_2(r_2)=R_4R_6$, $\phi_2(r_3)=R_3R_7$, $\phi_2(r_4)=R_1R_8$, $\phi_2(e_1)=E_2E_5$,
   $\phi_2(e_2)=E_4E_6$, $\phi_2(e_3)=E_3E_7$ and  $\phi_2(e_4)=E_1E_8$. Furthermore $\Br(\ddH_4)$ are free of rank $236025$ over
   $\Z[\delta^{\pm 1}]$.
 \end{thm}
\section{Definitions}
Let $\delta$ be a generator of a infinite cyclic group and  $\Z[\delta^{\pm 1}]$ be the group algebra over $\Z$ for the infinite cyclic groups.
\begin{defn}\label{0.1} For $k=3$, $4$,
the Brauer algebra of type $\ddH_k$, denoted by $\Br(\ddH_k)$,
is a unital associative $\Z[\delta^{\pm 1}]$-algebra generated by $\{r_i,\, e_i\}_{i=1}^{k}$,  subject to
the following relations.
\begin{eqnarray}
\delta\delta^{-1}&=&1\label{0.1.1}
\\
\delta x&=&x\delta  \ \mbox{for each generator  } x \label{0.1.2}
\\
r_{i}^{2}&=&1 \qquad \qquad\,\,\,\kern.02em \mbox{for}\,\mbox{any} \ i   \label{0.1.3}
\\
r_ie_i &= & e_ir_i \,=\, e_i \,\,\,\,\,\,\kern.05em \mbox{for}\,\mbox{any}\ i  \label{0.1.4}
\\
e_{i}^{2}&=&\delta^2 e_{i} \qquad \quad\,\,\kern.02em \mbox{for}\,\mbox{any}\ i     \label{0.1.5}
\\
r_ir_j&=&r_jr_i, \qquad \quad \mbox{for}\ i\nsim j   \label{0.1.6}
\\
e_ir_j&=&r_je_i, \qquad  \quad \kern-.03em \mbox{for}\ i\nsim j     \label{0.1.7}
\\
e_ie_j&=&e_je_i, \qquad \quad \kern-.06em \mbox{for}\ i\nsim j      \label{0.1.8}
\\
r_ir_jr_i&=&r_jr_ir_j, \qquad\,\kern-.04em \mbox{for}\ {i\sim j}\,   \label{0.1.9}
\\
r_jr_ie_j&=&e_ie_j , \quad \qquad \kern-.11em \mbox{for}\ i\sim j\,               \label{0.1.10}
\\
r_ie_jr_i&=&r_je_ir_j, \quad \quad \kern.06em \mbox{for}\ i\sim j\,           \label{0.1.11}
\\
r_1r_2r_1r_2 r_1&=&r_2r_1r_2r_1r_1,                                      \label{0.1.12}
 \\
r_1r_2e_1r_2 r_1&=&r_2r_1e_2r_1r_1,                                \label{0.1.13}
\\
  e_1e_2e_1&=&e_1,                                                        \label{0.1.14}
\\
e_1r_2e_1&=&e_1,                                                        \label{0.1.15}
\\
e_1r_2r_1r_2 e_1&=& e_1,                                      \label{0.1.16}
\end{eqnarray}
and one additional relation for $k=4$
\begin{eqnarray}
e_1(r_2r_1r_2r_1r_3r_2r_1r_2r_3r_1r_2r_1r_2r_3r_4)^5= e_1.   \label{0.1.17}
\end{eqnarray}
Here $i\sim j$ means that $i$ and $j$ are connected by a simple bond and $i\nsim j$ means that
there is no bond (simple or multiple) between $i$ and $j$ in the Coxeter Diagram of type $\ddH_k$ depicted in Figure 1.
 The submonoid
of the multiplicative monoid of $\Br(\ddH_k)$
generated by $\delta$, $\{r_i,\,e_i\}_{i=1}^{k}$   is
denoted by $\BrM(\ddH_k)$. This is the monoid of monomials in
$\Br(\ddH_k)$.
\end{defn}
If we just focus on  relations between $\{r_1, r_2, e_1, e_2\}$, this will give the algebra of $\Br(\ddI_{2}^5)$ in \cite{L20132},
which is also  isomorphic to  the algebra $\ddB_{\ddG_5}(\gamma)$ in \cite{ZhiChen} up to some parameters.  \\

We can recall from \cite{CFW2008} the definition of Brauer algebra  $\Br(Q)$ of simply laced types of   a graph $Q$
 defined as an associative algebra over $\Z[\delta^{\pm 1}]$ with a Coxeter group generator $R_i$ and a
 Temperley-Lieb generator $E_i$ associated to each vertex $i$ of $Q$,   subject to the  relation (\ref{0.1.1})-(\ref{0.1.11}) by
  replacing the $\delta^2$ in (\ref{0.1.5}) with $\delta$. naturally the monoid generated by $\delta$, $R_i$ and $E_i$ is called the Brauer monoid of
 type $Q$, denoted by $BrM(Q)$.  For each $Q$, the algebra $Br(Q)$ is free over $\Z[\delta^{\pm 1}]$. The classical
 Brauer algebra on $m+1$ strands arises when $Q=\ddA_m$.
 \section{admissible root sets}
Let $\{\beta_i\}_{i=1}^{4}$ be   simple  roots of $W(\ddH_4)$ corresponding to the notation of Figure \ref{H3H4} ($\{\beta_i\}_{i=1}^{3}$ for $W(\ddH_3)$), and $\{r_i\}_{i=1}^{4}$ be the reflections corresponding to those simple roots.
They can be embedded into Euclidean space $\R^4$ with each of them having Euclidean length $\sqrt{2}$ and
$(\beta_1,\beta_2)=\frac{-\sqrt{5}-1}{2}=-\varphi$ with  $\varphi^2=\varphi+1$.   Let $\Psi_4$ and $\Psi_3$ denote the root systems of
$W(\ddH_4)$ and $W(\ddH_3)$ respectively and $\Psi_4^+$ and $\Psi_3^+$ the positive roots respecting $\{\beta_i\}_{i=1}^{4}$.
It is known $\#\Psi_3^+=15$ and $\#\Psi_4^+=60$. A mutually orthogonal subset $B\in \Psi_4^+$ ($\Psi_3^+$) is called  an orthogonal basis
if $B$ can span $\R_4$  ($\R_3$). And it is known that
any  mutually orthogonal subsets of the same cardinality are on the same orbit under $W(\ddH_4)$ ($W(\ddH_3)$) by GAP(\cite{GAP}).
Here we just consider that the natural Coxeter group action $W(\ddH_4)$ ($W(\ddH_3)$) on $\Psi_4$ ($\Psi_3$) are restricted to
the positive roots by negating the negative ones.
Let  $$\beta_5=\varphi^2\beta_1+2\varphi \beta_2+\varphi\beta_3=r_2r_1r_2r_1r_3r_2\beta_1\in \Psi_3,$$
and $r_5$ is the corresponding reflection of $\beta_5$ in $W(\ddH_3)$ and  $W(\ddH_4)$.
Let $N_1^3$ and $N_1^3$ be the stabilizers of $\beta_1$ in $W(\ddH_3)$ and  $W(\ddH_4)$ respectively.
\begin{lemma}\label{N13N14} We have that
\begin{eqnarray*}
N_1^3&=&\left<r_1, r_3,r_5\right>\cong W(\ddA_1)^3\\
N_1^4&=&\left<r_1, r_3, r_4, r_5\right>\cong W(\ddA_1)\times W(\ddH_3).
\end{eqnarray*}
 \end{lemma}
\begin{proof}
We have the following results of inner products involving $\beta_5$,\\
 $(\beta_1,\beta_5)=0=(\beta_3, \beta_5)$,  $(\beta_4,\beta_5)=-\varphi$. \\
 Hence we have the diagram relation for them as below which indicates that our claim about the
 group isomorphisms in this lemma holds. Using this observation, we can compute indices of the two subgroups $\left<r_1, r_3,r_5\right>$ and  $\left<r_1, r_3, r_4, r_5\right>$
 in  $W(\ddH_3)$ and $W(\ddH_4)$ and show that these are $15=\#\Psi_3^+$ and $60=\#\Psi_4^+$, respectively.
 $$\circ_{1}\quad \circ_{3}\rule{10mm}{.2mm}\circ_{4}\overset{5}{\rule{10mm}{.2mm}}\circ_5$$
By the diagram, we find that the two (resp. three) reflections  stabilize $\beta_1$ in  $W(\ddH_3)$
(resp. $W(\ddH_4)$), respectively.
Therefore the lemma follows from  Lagrange's Theorem.

 \end{proof}
 It can be verified that that the subgroup generated by  $(r_5r_3r_4)^5$ is a normal subgroup of
 $\left<r_3,r_4, r_5\right>\cong W(\ddH_3)$ and $(r_5r_3r_4)^5$ has order $2$.
 Let $C_1^3=\left< r_3, r_5\right>$  and $C_1^4$ be representatives of left coset of
 $\left<(r_5r_3r_4)^5\right>$ in $\left<r_3,r_4, r_5\right>$.
  Then
 $C_1^3\cong W(\ddA_1)^2$ and $C_1^4\cong W(\ddH_3)/\left<(r_5r_3r_4)^5\right>$. Let $D_1^3$ and $D_1^4$ be the left coset representatives
 of $N_1^3$ in $W(\ddH_3)$ and  $N_1^4$ in $W(\ddH_4)$ respectively. Then $\# D_1^3=\#\Psi_3^+=15$,
 $\# D_1^4=\#\Psi_4^+=60$.
 As in \cite{CLY2010} and \cite{CL2012}, we have the following lemma.
 \begin{lemma}\label{lem.ij} For $i$ and $j$ be nodes of the Coxeter diagram of $\ddH_3$ or $\ddH_4$. If $w\in W(\ddH_4)$ or $W(\ddH_3)$,
 satisfies $w\beta_i=\beta_j$, then $we_iw^{-1}=e_j$.
 \end{lemma}
 \begin{proof}It suffices to prove that any generator of $N_1^4$ in Lemma \ref{N13N14}, satisfies that $re_1r^{-1}=e_1$.
 The cases of  $r_1$, $r_3$, $r_4$  hold trivially; but for $r_5$,  we just apply the following formulas which can be deduced easily from the definition,
 \begin{eqnarray*}
 r_2r_1r_2r_1e_2r_1r_2r_1r_2&=&e_1,\\
  r_1r_2r_1r_2e_1r_2r_1r_2r_1&=&e_2,\\
  r_2r_3e_2r_3r_2&=&e_3,\\
  r_3r_2e_3r_2r_3&=&e_2,
 \end{eqnarray*}
 then we have
 \begin{eqnarray*}
 r_5e_1r_5&=&r_2r_1r_2r_1r_3r_2r_1r_2r_3(r_1r_2r_1r_2e_1 r_2r_1r_2r_1)r_3r_2r_1r_2r_3r_1r_2r_1r_2\\
 &=&r_2r_1r_2r_1r_3r_2r_1(r_2r_3e_2r_3r_2)r_1r_2r_3r_1r_2r_1r_2\\
 &=&r_2r_1r_2r_1r_3r_2(r_1e_3r_1)r_2r_3r_1r_2r_1r_2\\
 &=&r_2r_1r_2r_1(r_3r_2e_3r_2r_3)r_1r_2r_1r_2\\
 &=&r_2r_1r_2r_1e_2r_1r_2r_1r_2=e_1.
 \end{eqnarray*}
 \end{proof}
 We define $e_\beta=re_{i}r^{-1}$, if $\beta=r\beta_i$. By the above lemma it is well defined.
 Since Coxeter groups $W(\ddH_3)$ and $W(\ddH_4)$ acts transitively on mutually orthogonal root  sets of the same cardinality, hence $e_{\beta}e_{\beta'}=e_{\beta'}e_{\beta}$,
if $\{\beta,\beta'\}$ is a mutually orthogonal root set,  in view of  $e_1e_3=e_3e_1$. Therefore for  any   mutually orthogonal root subset $B$  of  $\Psi_3^+$ or
$\Psi_4^+$, we can  define that
$$e_{B}=\prod_{\beta\in B}e_{\beta}.$$
\begin{lemma}\label{psi3psi4}
For $\Psi_3^+$ and $\Psi_4^+$, The following holds.
\begin{enumerate}[(I)]
\item For each $\beta\in \Psi_3^+$, there is a unique orthogonal basis subset of $\Psi_3^+$ containing $\beta$,
hence there are $5$ different orthogonal basis subsets of $\Psi_3^+$.
\item For each $\beta\in \Psi_4^+$, there is $5$ orthogonal basis subsets of $\Psi_4^+$ containing it,
hence there are $75$ different orthogonal basis subsets of $\Psi_4^+$.
\item  For any orthogonal subset $B$ of $\Psi_4^+$ with $\#B>1$, there is a unique orthogonal basis subset of $\Psi_4^+$ containing $B$.
\end{enumerate}
\end{lemma}
\begin{proof} Any $\beta$ can be obtained from $\beta_1$ by acting some element in $W(\ddH_3)$, we can just consider
$\beta=\beta_1$.
If $\beta_1\in B\subset\Psi_3^+$ is an orthogonal basis, then there exist another two positive roots $\beta'$, $\beta''$ in $B$ being
orthogonal to $\beta_1$, which implies that the corresponding reflection can fix $\beta_1$, by the lemma \ref{N13N14}, we see that
$B$ must be $\{\beta_1, \beta_3, \beta_5\}$. Hence the number of orthogonal basis subsets of $\Psi_3^+$ is $\frac{15}{3}=5$.\\
The first conclusion of (II) follows from (I) and Lemma \ref{N13N14}. The second one holds as $\frac{60\times 5}{4}=75.$\\
The (III) follows from (II) and (I).
\end{proof}
\begin{defn} If a mutually orthogonal subset $B$ in $\Psi_3^+$  ($\Psi_4^+$) has at most one element or is an orthogonal basis,   then we call
$B$  an \emph{admissible root set} of type $\ddH_3$  ($\ddH_4$) or $B$ \emph{admissible}.
\end{defn}
\begin{rem} By Lemma \ref{N13N14} and Lemma \ref{psi3psi4}, there exists a unique positive root  for $C_1^4\cong W(\ddH_3)$  with  $\beta_3$,  $\beta_4$, and $\beta_5$ being the simple roots which are orthogonal to $\beta_3$, $\beta_5$, and we denote it by $\beta_7$.
Hence the $\{\beta_1,\beta_3,\beta_5,\beta_7\}$ is an orthogonal basis and admissible root set of type $\ddH_4$.
\end{rem}
By Lemma \ref{psi3psi4}, for each mutually orthogonal subset $B$ of $\Psi_3^+$ $(\Psi_4^+)$, there exists a minimal unique admissible root set
containing $B$, we denote it by $B^{\rm cl}$, and call it the \emph{admissible closure} of $B$.
As in \cite{CL2012},  the lemma below holds.
\begin{lemma}\label{cldelta}For each each mutually orthogonal subset $B$ of $\Psi_3^+$ $(\Psi_4^+)$, we have
$$e_{B^{\rm cl}}=\delta^{2\#(B^{\rm cl}\setminus B)}e_{B}.$$
\end{lemma}
Let $N_2^3$ and $N_2^4$ be stabilizers of $\{\beta_1, \beta_3, \beta_5\}$ in $W(\ddH_3)$ and
$\{\beta_1, \beta_3, \beta_5\, \beta_7\}$ in $W(\ddH_4)$ respectively. Let $D_2^3$ and $D_2^4$ be the left coset representatives
 of $N_2^3$ in $W(\ddH_3)$ and  $N_2^4$ in $W(\ddH_4)$ respectively. Then $\# D_2^3=5$,
 $\# D_2^4=75$ by Lemma \ref{psi3psi4}.
 \section{Normal forms of $\BrM(\ddH_3)$ and $\BrM(\ddH_4)$}
 As in \cite{CLY2010}, the following conclusion can hold by easy verification.
 \begin{lemma} \label{grpelements}
 In $\BrM(\ddH_3)$ $(\BrM(\ddH_4))$, the following holds.
 \begin{enumerate}[(i)]
 \item Each element in $C_1^3$ $(C_1^4)$  commutes with $e_1$.
 \item For each element $r\in W(\ddH_3)$ $(W(\ddH_4))$, there exist  $r'\in D_1^3$ $(D_1^4)$ and $r''\in C_1^3$ $(C_1^4)$, such that
 $$r e_1= r' e_1 r''. $$
 \item For each element $r\in W(\ddH_3)$ $(W(\ddH_4))$, there exists an element $r'\in D_2^3$ $(D_2^4)$,  such that
 $$r e_1e_3= r'e_1e_3. $$
 \end{enumerate}
 \end{lemma}
 Next we consider the cases for Temperley-Lieb elements(\cite{TL1971}).
 \begin{lemma}\label{ebetae1} Let $\beta\in \Psi_3^+$ $(\Psi_4^+)$, such that $\beta$ is not equal to or orthogonal to $\beta_1$. Then there exists some
   $r'\in D_1^3$ $(D_1^4)$ and $r''\in C_1^3$ $(C_1^4)$, such that
$$e_{\beta}e_1=r' e_1 r''.$$
 \end{lemma}
 \begin{proof}
  If $(\beta_1, \beta)=\pm 1$, since there exists such an element $r$ in $W(\ddH_3)$ or $W(\ddH_4)$ that
 $r\{\beta_1,\beta\}=\{\beta_2,\beta_3\}$; then $e_{\beta}e_1=r_1 r_{\beta}e_{1}$ in view of (\ref{0.1.10}). Consequently, the lemma follows from Lemma \ref{grpelements}.
 If $(\beta_1, \beta)\neq \pm 1$, then $\beta$ can be obtained by letting some element from $N_1^3$ or
 $N_1^4$ act on some root of  the linear combination of $\beta_1$ and $\beta_2$; hence it can be reduced to the case of $\BrM(\ddI_2^5)$ which has been
 verified in \cite{L20132}. Therefore the lemma also holds under this condition.
 \end{proof}
 By similar argument, the corollary below holds.
 \begin{cor}\label{ebetae1e3}
  Let $\beta\in \Psi_3^+$ $(\Psi_4^+)$. Then up to some power of $\delta$, there exists some
   $r'\in D_2^3$ $(D_2^4)$, such that
$$e_{\beta}e_1e_3=r' e_1e_3.$$
 \end{cor}
 \begin{proof}We consider the  corollary for $\Br(\ddH_4)$, and the conclusion for
 $\Br(\ddH_3)$ can be proved similarly.
  By Lemma \ref{cldelta}, we know that $e_1e_3e_{\beta_5}e_{\beta_7}=\delta^4 e_1e_3$.
 No $\beta\in \Psi_4^+$ can be orthogonal to  $\{\beta_1, \beta_3,\beta_5,\beta_7\}$.
 If $\beta\in \{\beta_1, \beta_3,\beta_5,\beta_7\}$, then $e_{\beta}e_1e_3=\delta^2e_1e_3.$
 If $\beta$ is not in
 $\{\beta_1, \beta_3,\beta_5,\beta_7\}$, there exists one $\beta_i$, $i\in \{1,3,5,7\}$, such that $\left<r_\beta,r_{\beta_i}\right>$ is isomorphic to $W(\ddI_2^5)$ or $W(\ddA_2)$; under the conjugation of $W(\ddH_4)$, the corollary follows from \cite{L20132} and (\ref{0.1.10}).
 \end{proof}
 \begin{thm}\label{rewritingforms} For each element in $\BrM(\ddH_k)$ for $k=3$, $4$, up to some power of $\delta$, it can be written as the three forms below,
 \begin{enumerate}[(I)]
 \item $r\in W(\ddH_k)$,
 \item $ue_1v w $, $u\in D_1^k$, $w^{-1}\in D_1^k$, $v\in C_1^k$,
 \item $ue_1e_3 w $, $u\in D_2^k$, $w^{-1}\in D_2^k$.
 \end{enumerate}
 \end{thm}
 \begin{proof}
 As in\cite{CFW2008}, \cite{CLY2010} and \cite{CL2012}, we can define a natural involution on $\Br(\ddH_k)$, by reversing the each monomial in $\BrM(\ddH_k)$, denoted
 by $\cdot^{\rm op}$.
 It can be verified that $\cdot^{\rm op}$ induces a natural isomorphism on $\Br(\ddH_k)$ and the difficult one  is (\ref{0.1.17}). By Lemma \ref{lem.ij} and Definition \ref{0.1}, we know that
 $e_1$ commutes with $r_3$, $r_4$, $r_5$. Since $(r_5r_3r_4)^5$ has order 2 we have that
 $((r_5r_3r_4)^5)^{\rm op}= (r_5r_3r_4)^5$. Then the left side of (\ref{0.1.17}) can be written as
 $$e_1(r_5r_3r_4)^5=(r_5r_3r_4)^5e_1=(r_5r_3r_4)^5)^{\rm op}e_1= (e_1(r_5r_3r_4)^5)^{\rm op}.$$ Consequently,
 the  equality (\ref{0.1.17}) still holds after application of  $\cdot^{\rm op}$.

 Note that $1$ is a normal form. To prove the theorem,  it suffices to
 prove that the above forms are closed under  left multiplication by $r_i$ for $i=1$, $\ldots$, $k$ and by  $e_{\beta}$ for $\beta\in \Psi_k^+$. Hence the lemma
 follows from Lemmas \ref{cldelta}, \ref{grpelements}, \ref{ebetae1} and Corollary \ref{ebetae1e3}.
\end{proof}
 \begin{rem} The numbers of the normal forms in the above are
 \begin{eqnarray*}
 120+15^2\times 4 +5^2&=&1045,\\
 14400+60\times 60^2+75^2&=&236025
 \end{eqnarray*}
 for $\Br(\ddH_3)$ and $\Br(\ddH_4)$ respectively.
 \end{rem}
 \section{Images of $\phi_1$ and $\phi_2$}

 To prove both of them are injective, we need to recall some results from \cite{CFW2008}, \cite{CGW2009}.
 Let  $\{\alpha_i\}_{i=1}^{6}$ be the simple roots of $W(\ddD_6)$ (Weyl group of type $\ddD_6$) corresponding to the diagram of $\ddD_6$ in figure \ref{H3H4}, and let
 $\Phi_6^+$ be the positive root of  $W(\ddD_6)$.
 From \cite[Proposition 4.9, Proposition 4.1]{CFW2008}, up to some power of $\delta$,
  there is a unique normal  form
 associated to  admissible roots of type $\ddD_6$,
 which are the orbits of $B_0^3=\emptyset$, $\{\alpha_2\}$, $B_1^3=\{\alpha_2, \alpha_4\}$, $\{\alpha_2, \alpha_4, \alpha_6\}$
 and $B_2^3=\{\alpha_1,\alpha_2, \alpha_4,\alpha_4^*,  \alpha_6,\alpha_6^*, \}$  for each element of $\BrM(\ddD_6)$,
where $\alpha^*$ is the orthogonal mate (see in \cite{CGW2009}) for each positive root $\alpha$ of type $\ddD_6$.
Now we prove the injectivity of $\phi_1$ and Theorem \ref{mainthm1} by  analyzing  the image of each form in Theorem \ref{rewritingforms}.
\begin{proof} It can be verified as in \cite{CLY2010} and \cite{L20132} that $\phi_1$  is an algebra homomorphism.
By \cite{M1992}, the normal forms in (I) of Theorem \ref{rewritingforms} is embedded into $\BrM(\ddD_6)$ by $\phi_1$.\\
It can be verified that  that $\phi_1(r_5)=R_4^*R_6^*$, where $R_4^*$ and $R_6^*$ are reflections corresponding to $\alpha_4^*$ and
$\alpha_6^*$. By the diagram representations for $\Br(\ddD_6)$ in \cite{CGW2009}, $\phi_1(C_1^{3})=\left<R_1R_6,R_4^*R_6^* \right>$ is embedded into $W(C_{WB_1^3})$, the commutator subgroup of $B_1^3$(\cite{CFW2008}).
Since $\phi_1(e_1)=E_2E_4$, and $\phi_1(W(\ddH_3))B_1^3$ has $15$ different subsets of $\Phi_6^+$, therefore the   normal forms in (II) of Theorem \ref{rewritingforms} is embedded into the cell associated to $\{\alpha_2,\alpha_3\}$. \\
The   normal forms in (III) of Theorem \ref{rewritingforms} is embedded  into the cell associated to $B_2^3$, for $\phi_1(e_1e_3e_{\beta_5})=E_{B_2^3}$, and $\phi_1(W(\ddH_3))B_2^3$ has $5$ different subsets of $\Phi_6^+$.\\
Therefore $\phi_1$ is an injective homomorphism.
\end{proof}
Before proving the Theorem \ref{mainthm2}, we need to recall some results in \cite{CW2011}.
We keep  notation as  in \cite[Section 2]{CW2011} and first introduce some basic concepts.
Let $M$ be  the  diagram  of a  connected finite simply laced Coxeter
group (type $\ddA$, $\ddD$, $\ddE_6$,  $\ddE_7$,  $\ddE_8$).
$\BrM(M)$ is the associated Brauer monoid as in \cite{CFW2008}.
An element  $a\in \BrM(M)$ is said to be of \emph{height} $t$ if the minimal number of
 $R_i$  occurring in an expression of $a$ is $t$, denoted by $\rm{ht}$$(a)$. By $B_Y$ we denote
the admissible closure (\cite{CFW2008}) of $\{\alpha_i|i\in Y\}$, where $Y$ is a coclique
of $M$. The set $B_Y$ is a minimal element in the $W(M)$-orbit of $B_Y$ which is endowed with a  poset  structure
induced by the partial ordering $<$ (\cite{CGW2006})
defined on $\cA$ (the set of all admissible sets). If $d$ is the Hasse diagram distance for $W(M)B_Y$
from $B_Y$ to the unique maximal element
(\cite[Corollary 3.6]{CGW2006}), then for $B\in W(M)B_Y$ the height of $B$,
notation $\rm{ht}$$(B)$, is $d-l$, where $l$ is the distance in the
Hasse diagram from $B$ to the maximal element. \\
In \cite{CFW2008}, a Brauer  monoid action is defined as  follows.
For any mutually orthogonal positive root set $B$, we define $B^{\rm cl}$ to be  the \emph{admissible closure} of
$B$, namely the minimal admissible root set(\cite{CFW2008}) containing $B$.
The generator $R_{i}$ acts
by the natural action of Coxeter group elements on its
root sets, where negative roots are negated so as to obtain positive roots,
the element $\delta$ acts as the identity,
and the action of $\{E_{i}\}_{i=1}^{n+1}$ is defined below.
\begin{equation}
E_i B :=\begin{cases}
B & \text{if}\ \alpha_i\in B, \\
(B\cup \{\alpha_{i}\})^{\rm cl} & \text{if}\ \alpha_i\perp B,\\
R_\beta R_i B & \text{if}\ \beta\in B\setminus \alpha_{i}^{\perp}.
\end{cases}
\end{equation}

By the natural involution, we can define a right monoid action of $\BrM(M)$ on $\cA$. \\
Considering our $\Br((\ddH_4)$ and table 3  in \cite{CW2011},
let  $$Y\in \mathcal{Y}=\{\emptyset, \{2\}, \{2,5\}, \{2, 3, 5,7\} \}$$
for type $\ddE_8$.
From \cite[Theorem 2.7]{CW2011}, it is known that each monomial  $a$ in $\BrM(\ddE_8)$ can be
uniquely written as $\delta^{i} a_{B} \hat{e}_Y h a_{B'}^{\rm op}$ for some $i\in \Z$ and $h\in W(M_{Y})$ in \cite[table 3]{CW2011},
where $B=a\emptyset$, $B^{'}=\emptyset a $, $a_{B}\in \BrM(\ddE_8)$, $a_{B'}^{\rm op}\in \BrM(\ddE_8)$ and
\\ (i) $a\emptyset=a_{B}\emptyset=a_{B}B_Y$,  $\emptyset a= \emptyset a_{B'}^{\rm op}= B_Y a_{B'}^{\rm op}$,
\\(ii) $\rm{ht}$$(B)=$\rm{ht}$(a_{B})$, $\rm{ht}$$(B')=$\rm{ht}$(a_{B'}^{\rm op})$. \\
In view of $E_5E_6E_7E_2E_4E_5\{\alpha_6, \alpha_8\}=\{\alpha_2, \alpha_5\}$ and $\rm ht$$E_5E_6E_7E_2E_4E_5=0$, hence
\begin{eqnarray*}
&&W(M_{\{2,5\}})=E_5E_6E_7E_2E_4E_5W(M_{\{6,8\}})E_5E_4E_2E_7E_6E_5\cong W(\ddA_5)\\
&&=\left<\hat{E}_2\hat{E}_5R_3,\hat{E}_2\hat{E}_5R_1,\hat{E}_2\hat{E}_5R_8E_6E_4E_5E_7R_3E_4E_6E_5E_2,
 \hat{E}_2\hat{E}_5R_8,\hat{E}_2\hat{E}_5R_7 \right>,
\end{eqnarray*}
 where $\hat{E}_i=\delta^{-1} E_i$, and $W(M_{\{2,3,5,7\}})$ is the trivial group.
Let $R_1'=\hat{E}_2\hat{E}_5R_3$, $R_2'=\hat{E}_2\hat{E}_5R_1$, $R_3'=\hat{E}_2\hat{E}_5R_8E_6E_4E_5E_7R_3E_4E_6E_5E_2$,
 $R_4'=\hat{E}_2\hat{E}_5R_8$, $R_5'=\hat{E}_2\hat{E}_5R_7$. They can be considered the natural generators of
 $W(\ddA_5)$ with their indices by use of  \cite[Proposition 4.3]{CW2011}.
 Let $\alpha'=2\alpha_4+\alpha_2+\alpha_3+\alpha_5$ and $\alpha''=\alpha_2+\alpha_3+2\alpha_4+2\alpha_5+2\alpha_6+\alpha_7$.
Considering  the natural embedding of
 $W(\ddH_4)$ into $W(\ddE_8)$ through $\phi_2$, it can be checked that
 $\phi_2(r_5)=R_{\alpha'}R_{\alpha''}$, and  $\phi_2(e_1r_5)=E_2E_5R_{\alpha'}R_{\alpha''}=\delta^2 R_1'R_3'$.
After these preparations, now we can start to prove Theorem \ref{mainthm2}.
\begin{proof}  As in proving $\phi_1$ is an algebra homomorphism,  it can be verified that (\ref{0.1.1})-(\ref{0.1.16}) still hold under
$\phi_2$. In view of each element in $\{\alpha_1,\alpha_3, \alpha',\alpha'',\alpha_7, \alpha_8\}$ is orthogonal to
$\{\alpha_2, \alpha_5\}$, hence
\begin{eqnarray*}
\phi(e_1(r_5r_3r_4)^5)&=&E_2E_5(R_{\alpha'}R_{\alpha''}R_3R_7R_1R_8)^5\\
&=&\delta^2  (\hat{E}_2\hat{E}_5R_{\alpha'}\hat{E}_2\hat{E}_5R_{\alpha''}\hat{E}_2\hat{E}_5R_3\hat{E}_2\hat{E}_5R_7\hat{E}_2\hat{E}_5
R_1\hat{E}_2\hat{E}_5R_8)^5\\
&=&\delta^2(R_1'R_3'R_1'R_5'R_2'R_4')^5\\
&=&\delta^2(R_3'R_5'R_2'R_4')^5,
\end{eqnarray*}
applying the M\"{u}hlherr's partition for $W(\ddI_2^5)$ in $W(\ddA_4)$ (with generators $\{R_i'\}_{i=2}^5$), the above is
equals to $\delta^2 1_{W(M_{\{2,5\}})}=\delta^2 \hat{E}_2\hat{E}_5= E_2E_5=\phi_2(e_1)$, therefore the equality (\ref{0.1.17}) holds under $\phi_2$.\\
The following is dedicated to prove the injectivity of  $\phi_2$.
 By \cite{M1992}, the  normal  forms in (I) of  Theorem \ref{rewritingforms} is embedded into $W(\ddE_8)$
or the the cell associated to $\emptyset$ by $\phi_2$. \\
It can be checked that $\phi_2(W(\ddH_4))\{\alpha_2,\alpha_5\}\subset \cA$ has cardinality $60$, by the above we see that
there  is a group homomorphism
$C_1^4\rightarrow W(M_{\{2,5\}})$ defined by $x\rightarrow \delta^{-2}\phi_2(e_1x)$,
with the image being  the subgroup with cardinality $60$,  generated by $\{R_1'R_3', R_2'R_4', R_1'R_5'\}$ in the group $W(M_{\{2,5\}})$.
Since  $\#C_1^4=60$, therefore it is an isomorphism.   The  normal forms in (II) of Theorem \ref{rewritingforms} is embedded into the
cell in $W(\ddE_8)$  associated to $\{\alpha_2,\alpha_5\}$. \\
It can be also verified that $\phi_2(W(\ddH_4))\{\alpha_2,\alpha_5,\alpha_3,\alpha_7\}^{\rm cl}\subset \cA$ has cardinality $75$, hence  the  normal  forms of (III) in Theorem \ref{rewritingforms} is embedded into the
cell  in $W(\ddE_8)$  associated to $\{\alpha_2,\alpha_5,\alpha_3,\alpha_7\}^{\rm cl}$.
\end{proof}
\begin{rem} The isomorphism from $C_1^4\cong W(\ddH_3)/\left<(r_3r_5r_4)^5\right>$ to $W(M_{\{2,5\}})\cong W(\ddA_2)$, can be considered induced from the composition of canonical
homomorphism $$W(\ddH_4)\overset{g}{\longrightarrow} W(\ddD_6)\overset{f}{\longrightarrow} W(\ddA_5)$$ with $\left<(r_3r_5r_4)^5\right>$ being the kernel of $fg$.
\end{rem}
Just as in \cite{BO2011} and \cite{CL2012}, the  theorem below  about the cellularity can be obtained by the similar argument.
\begin{thm} If $R$ is a field such that   the group rings $R[W(H_3)]$ and  $R[W(H_4)]$  are cellular algebra (\cite{Gra} and \cite{GL1996}), then
the algebras $\Br(\ddH_3)\otimes R$ and $\Br(\ddH_4)\otimes R$ are  cellularly stratified algebras (\cite{HHKP2010}).
\end{thm}
\begin{rem} Up to some coefficients, $\Br(\ddH_3)$ is isomorphic to $\Br_{\ddG_{\ddH_3}}(\gamma)$. Up to some coefficients, we can obtain  $\Br_{\ddG_{\ddH_4}}(\gamma)$ by (\ref{0.1.1})-(\ref{0.1.16}) in Definition \ref{0.1}, which can be proved to have  rank  $452025$
by modifying the $C_1^4$ to  $\left<r_3,r_4, r_5\right>$ and the analogue  argument as in \cite{ZhiChen} for  $\Br_{\ddG_{\ddH_3}}(\gamma)$. Furthermore, for $\Br(\ddH_3)$, a faithful diagram representation can be 
obtained through the diagram representation of Brauer algebra of type $\ddD_n$ from 
\cite{CGW2009}.
\end{rem}

\end{document}